\documentclass[a4paper,11pt]{amsart}
\usepackage{amssymb}
\usepackage{amsmath}
\usepackage{latexsym}
\usepackage{eepic}
\usepackage{epsfig}
\usepackage{graphicx}
\usepackage{amscd}
\usepackage{eufrak}
\usepackage{mathrsfs}
\usepackage[english]{babel}
\usepackage{color}

\DeclareMathAlphabet{\mathpzc}{OT1}{pzc}{m}{it}
\newtheorem{theorem}[equation]{Theorem}
\newtheorem*{theorem*}{Theorem}
\newtheorem{theorem-definition}[equation]{Theorem-Definition}
\newtheorem{lemma-definition}[equation]{Lemma-Definition}
\newtheorem{definition-prop}[equation]{Proposition-Definition}

\newtheorem*{prop*}{Proposition}
\newtheorem{lemma}[equation]{Lemma}
\newtheorem{cor}[equation]{Corollary}
\newtheorem{definition}[equation]{Definition}
\newtheorem*{definition*}{Definition}

\newtheorem*{conjecture*}{Conjecture}
\newtheorem*{claim}{Claim}

\theoremstyle{definition}
\newtheorem{exam}[equation]{Example}

\newtheorem{remark}[equation]{Remark}

\newcommand{\LL}{\ensuremath{\mathbb{L}}}

\newcommand{\Z}{\ensuremath{\mathbb{Z}}}
\newcommand{\Q}{\ensuremath{\mathbb{Q}}}

\newcommand{\R}{\ensuremath{\mathbb{R}}}
\newcommand{\C}{\ensuremath{\mathbb{C}}}

\newcommand{\A}{\ensuremath{\mathbb{A}}}

\newcommand{\cX}{\ensuremath{\mathscr{X}}}
\newcommand{\mX}{\ensuremath{\widehat{\cX}}}

\newcommand{\cU}{\ensuremath{\mathscr{U}}}

\newcommand{\cY}{\ensuremath{\mathscr{Y}}}

\newcommand{\cC}{\ensuremath{\mathscr{C}}}

\renewcommand{\R}{\ensuremath{\mathbb{R}}}
\renewcommand{\C}{\ensuremath{\mathbb{C}}}

\renewcommand{\A}{\ensuremath{\mathbb{A}}}

\renewcommand{\cU}{\ensuremath{\mathscr{U}}}

\renewcommand{\cY}{\ensuremath{\mathscr{Y}}}

\newcommand{\Spec}{\ensuremath{\mathrm{Spec}\,}}

\newcommand{\ord}{\mathrm{ord}}
\newcommand{\red}{\mathrm{red}}

\newcommand{\an}{\mathrm{an}}

\newcommand{\wt}{\mathrm{wt}}
\newcommand{\lct}{\mathrm{lct}}
\newcommand{\Sk}{\mathrm{Sk}}
\newcommand{\spe}{\mathrm{sp}}

\newcommand{\llbr}{[\negthinspace[}
\newcommand{\rrbr}{]\negthinspace]}
\newcommand{\llpar}{(\negthinspace(}
\newcommand{\rrpar}{)\negthinspace)}

\numberwithin{equation}{section}

\newcommand{\sss}{\vspace{5pt} \subsection*{ }\refstepcounter{equation}{{\bfseries(\theequation)}\ }}

\newenvironment{theoremtag}[2]{\vspace{5pt} \noindent {\bf Theorem #1.} {\em #2}}{}

\begin{document}
\title[Poles of maximal order of motivic zeta functions]{Poles of maximal order of motivic zeta functions}

\author{Johannes Nicaise}
\address{KU Leuven\\
Department of Mathematics\\
Celestijnenlaan 200B\\3001 Heverlee \\
Belgium} \email{johannes.nicaise@wis.kuleuven.be}

\author{Chenyang Xu}
\address{Beijing International Center of Mathematics Research\\ Beijing University \\ Beijing \\ China}
\email{cyxu@math.pku.edu.cn}

\begin{abstract}
We prove a 1999 conjecture of Veys, which says that the opposite of the log canonical threshold is the only possible pole of maximal order of Denef and Loeser's motivic zeta function associated with a germ of a regular function on a smooth variety over a field of characteristic zero.
 We apply similar methods to study the weight function on the Berkovich skeleton associated with a degeneration of Calabi-Yau varieties.
 Our results suggest that the weight function induces a flow on
  the non-archimedean analytification of the degeneration towards the Kontsevich-Soibelman skeleton.
\end{abstract}

\maketitle

\section{Introduction}
\sss \label{sss-introf} Let $k$ be a field of characteristic zero  and set $R=k\llbr t\rrbr$ and  $K=k\llpar t \rrpar$. We endow $K$ with its $t$-adic absolute value
 $|x|=\exp(-\mathrm{ord}_tx)$.
 Let $X$ be a connected smooth $k$-variety and let $$f:X\to \Spec k[t]$$ be a non-constant regular function on $X$. We set $\cX=X\times_{k[t]}R$ and we denote by $\widehat{\cX}_K$ the generic fiber of the $t$-adic completion $\widehat{\cX}$ of $\cX$;
 this is a smooth $K$-analytic space. In \cite{MuNi}, Musta\c{t}\u{a} and the first-named author defined the {\em weight function}
  $$\wt_{(f)}:\widehat{\cX}_K\to \R\cup \{+\infty\}$$ that measures the singularities of the zero locus of $f$.
  It is closely related to the thinness function of \cite{BoFaJo}  and the log discrepancy function of \cite{JoMu}.
  If $v$ is the divisorial point of $\widehat{\cX}_K$ associated with a prime divisor $E$ on a birational modification of $X$, then $\wt_{(f)}(v)=\nu_E/N_E$ where $\nu_E$ is the log discrepancy of $E$ with respect to the pair $(X,(f))$ and $N_E$ is the vanishing order of $f$ along $E$. The minimal value of $\wt_{(f)}$ on $\widehat{\cX}_K$ is precisely the log canonical threshold of $f$.
        Every log resolution $h:X'\to X$ of $f$ gives rise to a Berkovich skeleton in $\widehat{\cX}_K$ that is canonically homeomorphic to the dual complex of the strict normal crossings divisor  $(f\circ h)$ on $X'$. The weight function $\wt_{(f)}$ is affine on each face of this skeleton. We will apply techniques from the Minimal Model Program (MMP) to prove that, if $\wt_{(f)}$ is constant on a maximal face of the Berkovich skeleton, then its value  is equal to the log canonical threshold of $f$. To be precise, this property holds only locally over $X$; we refer to Theorem \ref{thm-main} for the exact statement.

\sss This result has interesting consequences for the so-called {\em motivic zeta function} $Z_{f,x}(s)$ of $f$ at a closed point $x$ in the zero locus of $f$ on $X$. This is a rich invariant of the singularity of $f$ at $x$ that was defined by Denef and Loeser using motivic integration (see \cite{DL-barc} for a nice introduction). The motivic zeta function is a rational function over a suitable coefficient ring, and it is a longstanding problem to understand the nature of its poles (or the poles of closely related invariants, such as the topological zeta function or Igusa's $p$-adic zeta function). The Monodromy Conjecture predicts that every pole of the motivic zeta function is a root of the Bernstein polynomial of $f$.

   We denote by $\lct_x(f)$ the log canonical threshold of $f$ at $x$. The function $Z_{f,x}(s)$ has an explicit expression in terms of the geometry of the log resolution $h$, and this expression implies that the order of each pole is at most $n=\dim(X)$.
 Moreover, it is not difficult to deduce that $Z_{f,x}(s)$ has a pole at $s=-\lct_x(f)$, and that this is its largest pole.
   Veys conjectured in \cite{LaVe} that, if the topological zeta function of $f$ has a pole of order $n$, then this pole is the largest pole of the topological zeta function. We will deduce from Theorem \ref{thm-main} the following stronger form of Veys's conjecture.

  \begin{theoremtag}{\ref{thm-maxorder} (Veys's Conjecture)}{ If $s_0$ is a pole of order $n$ of the motivic zeta function $Z_{f,x}(s)$, then
   $s_0=-\lct_x(f)$. In particular, $s_0$ is a root of the Bernstein polynomial of $f$.}
\end{theoremtag}

  \smallskip
\noindent This statement implies the original conjecture of Veys because the order of a rational number $s$ as a pole of the motivic zeta function is at least the order of $s$ as a pole of the topological zeta function, since the latter is a specialization of the former.

\sss  Theorem \ref{thm-main} has an interesting counterpart for degenerations of Calabi-Yau varieties.
   The important ingredients in the definition of the weight function $\wt_{(f)}$ in \cite{MuNi} are the smooth $K$-analytic space $\widehat{\cX}_K$ and a volume form on $\widehat{\cX}_K$, a so-called Gelfand-Leray form, that is constructed (locally) from a volume form on $X$ (see \cite[\S6.3]{MuNi}). Thus it is natural to look at other situations where a smooth $K$-analytic space comes equipped with a volume form. We replace the $k[t]$-scheme $X$ from \eqref{sss-introf} by a geometrically connected smooth projective $K$-variety $X$ with trivial canonical sheaf. Let $\omega$ be a  volume form on $X$.
 We denote by $X^{\an}$ the Berkovich analytification of $X$; this $K$-analytic space will play the role of $\widehat{\cX}_K$. In \cite{MuNi}, Musta\c{t}\u{a} and the first-named author defined the weight function
  $$\wt_{\omega}:X^{\an}\to \R\cup\{+\infty\}$$
that measures the degeneration of $X$ at $t=0$. The locus where $\wt_{\omega}$ reaches its minimal value is independent of $\omega$.
 It is called the {\em essential skeleton} of $X$ and denoted by $\Sk(X)$. The essential skeleton is a non-empty compact subspace of $X^{\an}$ with a canonical piecewise integral affine structure. This object first appeared in Kontsevich and Soibelman's non-archimedean interpretation of Mirror Symmetry \cite{KS}.

 The essential skeleton can be computed as follows. Let $\cX$ be a regular proper $R$-model of $X$ whose special fiber $\cX_k$ is a divisor with strict normal crossings. Then there exists a canonical embedding of the dual complex of $\cX_k$ in $X^{\an}$. The image of this embedding is called the Berkovich skeleton of $\cX$ and denoted by $\Sk(\cX)$.
  It follows from techniques introduced by Berkovich \cite{berk-contr} and Thuillier \cite{thuillier} that $\Sk(\cX)$ is a strong deformation retract of $X^{\an}$.
  The weight function $\wt_{\omega}$ can reach its minimal value only at points of $\Sk(\cX)$, and it is affine on every face of $\Sk(\cX)$. It follows that the essential skeleton $\Sk(X)$ is a union of faces of $\Sk(\cX)$ (see \cite[4.5.5]{MuNi}).
 We will prove the following analog of Theorem \ref{thm-main}.

   \begin{theoremtag}{\ref{thm-KS}}{If $\tau$ is a maximal face of $\Sk(\cX)$ and $\wt_{\omega}$ is constant on $\tau$ with value $w$, then $w$ must be equal to the minimal value of $\wt_{\omega}$ on $X^{\an}$. Thus $\tau$ is contained in the essential skeleton $\Sk(X)$.}
 \end{theoremtag}

  \smallskip

\sss  In \cite{NiXu} we proved that $\Sk(X)$ is equal to the Berkovich skeleton of any good minimal $dlt$-model of $X$ over $R$ (the kind of model produced by the MMP).
   We then deduced from the results in \cite{dFKX}, obtained by a detailed analysis of the steps in the MMP, that the essential skeleton $\Sk(X)$ is a strong deformation retract of $X^{\an}$. It seems plausible that one can use the weight function to create a natural flow on $X^{\an}$ in the direction of decreasing values of $\wt_{\omega}$, and use this flow to contract $X^{\an}$ onto the subspace $\Sk(X)$ where $\wt_{\omega}$ takes its minimal value. Theorem \ref{thm-KS} supports this strategy; further evidence is provided by the following result.

    \begin{theoremtag}{\ref{thm-CY}}{
    For every real number $w$ we denote by $\Sk(\cX)^{\leq w}$ the subcomplex of $\Sk(\cX)$ spanned by the vertices where the value of $\wt_{\omega}$ is at most $w$. Then there exists a collapse of $\Sk(\cX)$ to the essential skeleton $\Sk(X)$ which simultaneously collapses $\Sk(\cX)^{\leq w}$ to $\Sk(X)$ for all $w$ greater than the minimal value of $\wt_{\omega}$ on $X^{\an}$.
    In particular,   $\Sk(X)$ is a strong deformation retract of $\Sk(\cX)^{\leq w}$.}
      \end{theoremtag}

\smallskip
\noindent A collapse is a particular kind of strong deformation retraction on a cell complex; see \eqref{sss-locflow} for a precise definition. Theorem \ref{thm-CY} also has a counterpart for hypersurface singularities: the analog of the essential skeleton is introduced in Definition \ref{def-essentialsk}, and the local version of Theorem \ref{thm-CY} is stated in Theorem \ref{thm-locallevel}.

\sss The paper is organized as follows.
 In Section \ref{sec-max} we deduce from the MMP the main technical result needed to prove Veys's conjecture (Theorem \ref{thm-main}).  The proof of the conjecture is given in Section \ref{sec-conj} (Theorem \ref{thm-maxorder}). To keep the proof as accessible as possible we avoided the language of weight functions on Berkovich spaces in these sections, although this interpretation was an important guide to obtain the results. In Section \ref{sec-weight} we explain the relation with weight functions, we define the essential skeleton of a hypersurface singularity and we study the level sets of the weight function on Berkovich skeleta (Theorem \ref{thm-locallevel}).
 Section \ref{sec-CY} contains the analogous results for degenerations of Calabi-Yau varieties (Theorems \ref{thm-KS} and \ref{thm-CY}).

\subsection*{Acknowledgements} We are grateful to Mircea Musta\c{t}\v{a} and Wim Veys for helpful discussions, and to the referees for their careful reading of the manuscript and their valuable suggestions. A part of this joint work was done while JN
 was a member of the program {\em Model Theory, Arithmetic Geometry and Number Theory} at MSRI, Berkeley, and CX was a member of the trimester program {\it Algebraic Geometry} at Hausdorff Research Institute for Mathematics, Bonn. The authors would like to thank both institutes for their hospitality. JN is partially supported by ERC Starting Grant MOTZETA (project 306610). CX is partially supported by the grant  ``Recruitment Program of Global Experts''.

\section{Maximal intersections with equal weights}\label{sec-max}

\sss \label{sss-wt}  We fix a base field $k$ of characteristic zero. Let $X$ be a connected smooth $k$-variety of dimension $n$ and let $\Delta$ be an effective $\Q$-divisor on $X$.
 Let $v$ be a divisorial valuation on $X$ with center contained in $\Delta$. This means that $v$ is a real valuation on the function field $k(X)$ and
  there exist a birational morphism $h:Y\to X$ of $k$-varieties, with $Y$ normal, and a prime component $E$ of $h^*\Delta$ such that $v$ is a real multiple of the
  valuation $\ord_E$ associated with $E$.
  We denote by $N_E$ the multiplicity of
$h^*\Delta$ along $E$  and by $\nu_E-1$ the
multiplicity of the relative canonical divisor $K_{Y/X}$ along $E$.
We set
$$\wt_{\Delta}(v)=\frac{\nu_E}{N_E}$$ and we call this positive rational
number the weight of $\Delta$ at $v$.
 This definition only depends on $v$, and not on the choice of the model $Y$.
 Note that $\wt_{\Delta}(v)=\wt_{\Delta}(\ord_E)$. For this reason, we will often denote $\wt_{\Delta}(v)$ by $\wt_{\Delta}(E)$.

\sss We fix a point $x$ on $X$.
 The log canonical threshold of $(X,\Delta)$ at $x$ is defined as
$$\lct_{x}(X,\Delta)=\inf_v\{\wt_{\Delta}(v)\},$$
 where $v$ runs through the set of divisorial valuations on $X$ whose center lies in $\Delta$ and contains $x$.
  It is well known that, in order to compute this infimum, it suffices to let $v$ run through the set of divisorial valuations associated with the prime components
   of the total transform of $\Delta$ on some log resolution of $(X,\Delta)$.

\sss \label{sss-assum} Let $h:Y\to X$ be a log resolution of $(X,\Delta)$ that is an isomorphism over $X\setminus \Delta$.
 We  write $$h^*\Delta=\sum_{i\in I}N_iE_i, \quad K_{Y/X}=\sum_{i\in I}(\nu_i-1)E_i$$
 where the $E_i$ are the prime divisors in $h^*\Delta$. For every non-empty subset $I'$ of $I$, we set $E_{I'}=\cap_{i\in I'}E_i$.
 Let $J$ be a non-empty subset of $I$ and let $C$ be a connected component of $E_J$. We assume that the intersection $h^{-1}(x)\cap C$ is non-empty but $h^{-1}(x)\cap C\cap E_i$ is empty for every $i$ in $I\setminus J$. This happens, for instance, if $C$ is a point contained in $h^{-1}(x)$.
 Our main technical result is the following.

\begin{theorem}\label{thm-main}
 We keep the notations and assumptions of \eqref{sss-assum}. If  we assume that the value
$\wt_{\Delta}(E_j)$ is the same for all $j$ in $J$ and we denote this value by $w$, then
we have $w=\lct_x(X,\Delta)$.
\end{theorem}
\begin{proof}
 We write $\Delta$ as a sum $A+B$ of effective divisors without common components such that
 $wA\leq A_{\red}$ and either $wB> B_{\red}$ or $B=0$.
 We define a new divisor $\Delta'$ on $X$ by
$$\Delta'=wA+B_{\red}=\min\{w\Delta,\Delta_{\red}\}$$ where the minimum is taken componentwise.
 We will prove eventually that $B=0$ and $\Delta'=wA$ locally at $x$, but at this point we need to construct $\Delta'$ in this artificial way to ensure that it is a boundary.

  We set $Y_0=Y$ and $\Delta_0=h^{-1}_*(\Delta')+(K_{Y/X})_{\red}$.
  Then we run the relative MMP for the pair $(Y_0,\Delta_0)$ over $X$ with scaling of some ample divisor.
 Since ${\rm Supp}( \Delta_0)={\rm Supp}(h^*\Delta)$, we know that for sufficiently small $\varepsilon>0$, this is the same as running the relative MMP for the $klt$ pair $(Y_0,\Delta_0-\varepsilon h^*\Delta)$. Hence, it follows from \cite{BCHM} that this MMP terminates with a minimal model.

 The outcome is a series of birational maps
 $$Y=Y_0\dashrightarrow Y_1 \dashrightarrow \ldots \dashrightarrow Y_m$$ where each of the $Y_i$ is a $\Q$-factorial normal projective $X$-scheme and each of the birational maps
 is a map of $X$-schemes.  If we denote by $\Delta_i$ the pushforward of the divisor $\Delta_0$ to $Y_i$, then the pair $(Y_i,\Delta_i)$ is $dlt$ for every $i$, and $K_{Y_m}+\Delta_m$ is nef over $X$.

\begin{claim} The birational map  $Y_{0}\dashrightarrow Y_{m}$
  is an open embedding on some open neighbourhood of $C\cap h^{-1}(x)$ in $Y_{0}$.
\end{claim}

\begin{proof}
  Let $\ell$ be an element of $\{-1,0,\ldots,m-1\}$. We will prove by induction on $\ell$ that the map $Y_0\dashrightarrow Y_{\ell+1}$ is defined at every point of $C\cap h^{-1}(x)$ and that it is an open embedding
  on some open neighbourhood of $C\cap h^{-1}(x)$ in $Y_0$. This is trivial for $\ell=-1$, so that we may assume that $\ell\geq 0$ and that the property holds for $Y_0\dashrightarrow Y_{\ell}$.
   With a slight abuse of notation, we will again write $E_i$ for the pushforward of $E_i$ to $Y_{\ell}$, for every $i$ in $I$. We write $C_\ell$ for the image of $C$ in $Y_{\ell}$. Then it is clear from the assumptions in \eqref{sss-assum} and the induction hypothesis that $C_\ell$ is still a connected component of the intersection of the divisors $E_j$, $j\in J$ on $Y_\ell$, and that for every $i\in I\setminus J$, the divisor $E_i$ on $Y_\ell$ is disjoint from the fiber of $C_\ell$ over $x$.
    We
   write $\Delta^{=1}_\ell$ for the reduced divisor on $Y_\ell$ consisting of the
    components of multiplicity one in $\Delta_\ell$.

 Let $y$ be a point of $C_\ell$ lying over $x\in X$. The birational map $Y_{\ell}\dashrightarrow Y_{\ell+1}$ is either a divisorial contraction or a flip. In both cases, it
  is induced by an extremal ray $R$ of $\overline{NE}(Y_{\ell}/X)$ such that \begin{equation}\label{eq-contr}R\cdot (K_{Y_{\ell}}+\Delta_{\ell}) <0.\end{equation}
 We denote by $g:Y_{\ell} \to Z$ the contraction of $R$.
  Since the pair $(Y_\ell,\Delta_\ell)$ is $dlt$, its log canonical centers are precisely the connected components of subsets of the form $D_1\cap\cdots \cap D_r$
    where $D_1,\ldots,D_r$ are prime components of $\Delta^{=1}_\ell$.
  A special case of \cite[6.6]{Ambro} (see also \cite[Prop.~25]{dFKX}) tells us that the set $S$ of log canonical centers of
  $(Y_{\ell},\Delta_{\ell})$ intersecting  the fiber  $g^{-1}(g(y))$ has a unique minimal element. But $C_\ell$ is such a minimal element,  because no component $E_i$ with $i\in I\setminus J$ intersects the fiber of $C_\ell$ over $x$ (this fiber contains $g^{-1}(g(y))$ since $g$ is a morphism of $X$-schemes).

        Now suppose that $g$ contracts a curve passing through $y$; the class of any such curve generates the ray $R$.
   Then $E\cdot R=0$ for every prime component $E$ of $\Delta^{=1}_{\ell}$ that is not one of the components $E_j$ with $j\in J$.
   Otherwise, $E$ would meet $g^{-1}(g(y))$ and $S$ would have a minimal element contained in $E$, which is impossible
   since $E$ does not contain $C_\ell$.

  In particular, $E\cdot R=0$ for every component $E$ of $(\Delta_\ell)_{\red}-\sum_{j\in J}E_j$
  that is contracted on $X$ or contained in the strict transform of $B$.
 Denoting by $f$ the morphism $f:Y_\ell\to X$, we compute:
   \begin{eqnarray*}
R\cdot (K_{Y_{\ell}}+\Delta_{\ell})&=& R\cdot (K_{Y_{\ell}}+\Delta_{\ell}-f^*(K_X+w\Delta))
\\ &=& R\cdot (K_{Y_{\ell}/X}+(K_{Y_{\ell}/X})_{\red}+f^{-1}_*\Delta' -f^*(w\Delta))
\\ &=& R\cdot(\sum_{j\in J}(\nu_j-wN_j)E_j+f^{-1}_*(B_{\red}-wB))
\\ &=& 0.
  \end{eqnarray*}
This contradicts the inequality \eqref{eq-contr}.
  We conclude that $g$ cannot contract a curve through $y$. Therefore, $g$
  must be an open embedding on some open neighbourhood of $y$ in $Y_{\ell}$, by Zariski's Main Theorem. It follows that  $Y_{\ell}\dashrightarrow Y_{\ell+1}$ is an open embedding on some open neighbourhood of the fiber of $C_\ell$ over $x$.
\end{proof}

Using this result, we can finish the proof of Theorem \ref{thm-main}. We denote by $f$ the morphism $f:Y_m\to X$ and we
  write $E_j$ for the image of $E_j$ in $Y_m$, for every $j\in J$.
Now consider the divisor
$$D= K_{Y_m}+\Delta_m-f^*(K_X+w\Delta)$$
 on $Y_m$. This divisor is nef over $X$. We can write $D$ as
 $$D=D_{\mathrm{exc}}-f^{-1}_*(wB-B_{\red})$$ where the divisor $D_{\mathrm{exc}}$ is $f$-exceptional and $wB-B_{\red}$ is effective.
 The negativity lemma \cite[3.39]{kollar-mori} implies that $-D$ is effective and that the support of $D$ is a union of fibers
  of $f$. But for every $j$ in $J$, the multiplicity of $D$ along $E_j$ is equal to $\nu_j-wN_j=0$ and thus $f^{-1}(x)\cap \mathrm{Supp}(D)=\emptyset$.
  This means that locally around $x$, we have $B=0$ and $\Delta'=w\Delta$. It also follows that
  $$(Y_m,f^*(w\Delta)-K_{Y_m/X})=(Y_m,\Delta_m)$$
  over some open neighbourhood of $x$ in $X$.
   This pair is $dlt$ and $\Delta_m$ contains components of multiplicity one intersecting $f^{-1}(x)$
  (for instance, the components $E_j$
  with $j\in J$). Thus
  $w=\lct_x(X,\Delta)$.
\end{proof}

\sss \label{sss-LV} Now suppose that we are still in the situation of Theorem \ref{thm-main}, that $\Delta$ is an effective $\Z$-divisor on $X$,
 and that the cardinality of $J$ is equal to $n$, the dimension of $X$ (in this case, $C$ is a point).
  Then Laeremans and Veys proved in \cite[Thm.~3.2]{LaVe} by combinatorial arguments that the weight $w$ is of the form $1/N$ for some positive integer $N$.

\begin{remark}
 One can also deduce the result in \cite[Thm.~3.2]{LaVe} from the following property.  We refer to \cite[Def.~13]{dFKX} for the definition of a $dlt$-modification.
  Let $Z$ be a smooth $k$-variety and let $D_0$ be an effective $\Z$-divisor on $Z$. Denote by $c$ the log canonical threshold of
 the pair $(Z,D_0)$, and set
 $D= cD_0$. Let  $h: (Z',D') \to (Z,D)$ be a $dlt$-modification of $(Z,D)$, where $D'$ denotes the log pullback of $D$ to $Z'$. Assume that $(Z',D')$ has a log canonical center $z$ of dimension zero.
 Then $c=1/N$ for some positive integer $N$.

 This property can be proven by running a relative MMP for $(Z',E)$ over $Z$, where $E$ denotes the $h$-exceptional part of $D'$, and applying adjunction to a one-dimensional log canonical center through $z$ that is contracted by the MMP.
    Since we will not use this result, we omit the details of the proof.
\end{remark}

\section{Poles of motivic zeta functions}\label{sec-conj}
\sss Let $X$ be a connected smooth $k$-variety, let $x$ be a closed point on $X$, and let $f$ be a regular function on $X$ such that $f(x)=0$.
 Denef and Loeser defined the {\em motivic zeta function} $Z_{f,x}(s)$
  of the germ of $f$ at $x$, an invariant that measures the singularity of $f$ at the point $x$.  It is  a power series
 in $\LL^{-s}$ over a certain Grothendieck ring $\mathcal{M}_x^{\hat{\mu}}$ of $\kappa(x)$-varieties with an action of the profinite group scheme $\hat{\mu}$ of roots of unity over $k$.
  Here $\kappa(x)$ denotes the residue field of $X$ at $x$ and $\LL^{-s}$ should be viewed as a formal variable. The zeta function $Z_{f,x}(s)$ is obtained from the generating series $Z_f(T)$ defined in \cite[\S3.2]{DL-barc}
  by applying the base change morphism $\mathcal{M}_{X_0}^{\hat{\mu}}\to \mathcal{M}_{x}^{\hat{\mu}}$ to its coefficients and setting $T=\LL^{-s}$.
  Closely related invariants are the so-called {\em naive} motivic zeta function $Z^{\mathrm{naive}}_{f,x}(s)$, which is a power series in $\LL^{-s}$ over the Grothendieck ring $\mathcal{M}_x$
  of $\kappa(x)$-varieties without group action, and the topological zeta function $Z_{f,x}^{\mathrm{top}}(s)$, which is an element of the field of rational functions $\Q(s)$.

\sss \label{sss-formula} What is important for our purposes is that each of these zeta functions can be explicitly computed on a log resolution. Set $\Delta=(f)$ and let $h:Y\to X$ be a log resolution of $(X,\Delta)$ that is an isomorphism over $X\setminus \Delta$.
 We denote by $E_i,\,i\in I$ the prime divisors in $h^*\Delta$ and we write $h^*\Delta=\sum_{i\in I}N_iE_i$ and $K_{Y/X}=\sum_{i\in I}(\nu_i-1)E_i$. For every non-empty subset $J$ of $I$, we set $E_J=\cap_{j\in J}E_j$ and $E^o_J=E_J\setminus (\cup_{i\notin J}E_i)$.
 Then we have the following expressions for the zeta functions introduced above (see \cite[\S3.3 and \S3.4]{DL-barc}):
 \begin{eqnarray*}
 Z_{f,x}(s) &=& \sum_{\emptyset\neq J\subset I}(\LL-1)^{|J|-1}[\widetilde{E}^o_J\times_{X}x]\prod_{j\in J}\frac{\LL^{-\nu_j-N_js}}{1-\LL^{-\nu_j-N_js}},
\\ Z^{\mathrm{naive}}_{f,x}(s) &=& \sum_{\emptyset\neq J\subset I}(\LL-1)^{|J|}[E^o_J\times_{X}x]\prod_{j\in J}\frac{\LL^{-\nu_j-N_js}}{1-\LL^{-\nu_j-N_js}},
\\ Z^{\mathrm{top}}_{f,x}(s) &=& \sum_{\emptyset\neq J\subset I}\chi(E_J^o\times_X x)\prod_{j\in J}\frac{1}{N_js+\nu_j}.
 \end{eqnarray*}
 Here $\LL$ denotes the class of the affine line $\A^1_k$, $\widetilde{E}^o_J$ is a certain finite \'etale cover of $E^o_J$ with an action of the group scheme $\hat{\mu}$, and
 $\chi(\cdot)$ denotes the $\ell$-adic Euler characteristic (which coincides with the singular Euler characteristic for the complex topology if $k$ is a subfield of $\C$).

\sss  It is obvious
from these explicit formulas that each pole is of the form $\wt_{\Delta}(E_i)=-\nu_i/N_i$ for some $i\in I$ (see Remark \ref{rem:pole} for a precise definition of the poles).
 Thus the largest possible pole is the negative of the log canonical threshold
 $$\lct_x(X,\Delta)=\min\left\{\frac{\nu_i}{N_i}\,|\,i\in I,\,x\in h(E_i)\right\}$$
  of $f$ at $x$.
However, in practice most of these candidate-poles will not be actual poles due to
cancellations in the formulas. This phenomenon would be explained by Denef and Loeser's motivic monodromy conjecture,
 which predicts that every pole of each of these three zeta functions is a root of the Bernstein polynomial of $f$.
  This conjecture was motivated by an analogous conjecture of Igusa for  $p$-adic local zeta functions of polynomials over number fields.
   Recall that
   $-\lct_x(X,\Delta)$ is always the largest root of the Bernstein polynomial of $f$ at $x$; see for instance \cite[10.6]{kollar}.
 The monodromy conjecture has been proven if $\dim(X)=2$ \cite{loeser, Rod} and also for some special classes of singularities, but it remains wide open in general. We refer
  to \cite{ni-mon} for a gentle introduction and a survey of some known results.

\sss It is also clear from the formulas in \eqref{sss-formula} that the order of a pole is at most $n=\dim(X)$, since $E_J$ is empty for every subset $J$ of $I$ of cardinality strictly larger than $n$.
 In \cite[0.2]{LaVe}, Veys made the following conjecture.
 \begin{conjecture*}[Veys]
  If $Z^{\mathrm{top}}_{f,x}(s)$ has a pole $s_0$ of order $n$, then $s_0$ must be the largest pole of
  $Z^{\mathrm{top}}_{f,x}(s)$.
 \end{conjecture*}
 Veys proved this statement if $n=2$ \cite[4.2]{veys} and also if $f$ is a polynomial that is non-degenerate with respect to its Newton polyhedron \cite[2.4]{LaVe},
  but these were the only cases known so far.
  We can deduce from Theorem \ref{thm-main} the following refinement of Veys's conjecture.

\begin{theorem}\label{thm-maxorder}
Let $X$ be a connected smooth $k$-variety of dimension $n$, let $x$ be a closed point on $X$, and let $f$ be a non-constant regular function on $X$.
 Let $h:Y\to X$ be a log resolution for $f$ as in \eqref{sss-formula}, and denote by $m$ the largest positive integer such that there exists a subset $J$ of $I$ of cardinality $m$
 with $E_J\cap h^{-1}(x)\neq \emptyset$ and $\nu_j/N_j=\lct_x(X,\Delta)$ for every $j\in J$.
  Then the following properties hold.
 \begin{enumerate}
 \item \label{it:ispole} The motivic zeta functions $Z_{f,x}(s)$ and $Z^{\mathrm{naive}}_{f,x}(s)$ have a pole of order $m$ at $s=-\lct_x(X,\Delta)$, and this is their largest pole. If $m=n$, then the topological zeta function $Z^{\mathrm{top}}_{f,x}(s)$ has a pole of order $n$ at $s=-\lct_x(X,\Delta)$, and this is its largest pole.

 \item \label{it:maxord} Conversely, if $s_0$ is a pole of order $n$ of $Z_{f,x}(s)$, $Z^{\mathrm{naive}}_{f,x}(s)$ or $Z^{\mathrm{top}}_{f,x}(s)$, then
   $s_0=-\lct_x(X,\Delta)$ and $m=n$. Moreover, $s_0$ is of the form $-1/N$ for some positive integer $N$.
 \end{enumerate}
\end{theorem}
\begin{proof}
\eqref{it:ispole} This result is more or less folklore, and it can be proven by straightforward computation.
 Note that it is clear from the expressions in \eqref{sss-formula} that $Z_{f,x}(s)$, $Z^{\mathrm{naive}}_{f,x}(s)$ and $Z^{\mathrm{top}}_{f,x}(s)$ have no poles that are strictly larger than
  $-\lct_x(X,\Delta)$, and that the order of $-\lct_x(X,\Delta)$ as a pole is at most $m$.
 Now we specialize $Z_{f,x}(s)$ and $Z^{\mathrm{naive}}_{f,x}(s)$
 to elements in $\Z[u,u^{-1}][\LL^{-s}]$ by means of the ring morphisms
 $$\mathcal{M}_x^{\hat{\mu}}\to \mathcal{M}_x\to \Z[u,u^{-1}]$$
 where the first morphism simply forgets the $\hat{\mu}$-action and the second one sends the class of a $\kappa(x)$-variety $Z$ to the
 Poincar\'e polynomial $P_Z(u)$ of $Z$ (see \cite[\S8]{Ni-tracevar}). What matters here is that $P_Z(u)$ is a non-zero polynomial with positive leading coefficient
  if $Z$ is non-empty. Using this property, one easily verifies that the residue at the expected pole of order $m$ at $s=-\lct_x(X,\Delta)$ is different from zero.
   Likewise, if $m=n$, then one immediately sees that the residue of the expected pole of $Z^{\mathrm{top}}_{f,x}(s)$ of order $n$ at $s=-\lct_x(X,\Delta)$ is positive.

\eqref{it:maxord} If $s_0$ is a pole of order $n$,
 then it follows from  the explicit formulas for the zeta functions in \eqref{sss-formula} that there must exist a subset $J$ of $I$ of cardinality $n$ such that $E_J\cap h^{-1}(x)$ is non-empty and $s_0=-\nu_j/N_j$ for every $j$ in $J$. By Theorem
  \ref{thm-main}, this can only happen when $s_0=-\lct_x(X,\Delta)$ and $m=n$. As we mentioned in \eqref{sss-LV}, it was already shown in \cite{LaVe} that $s_0$ is of the form $-1/N$.
\end{proof}

\sss In particular, a pole of order $n$ of $Z_{f,x}(s)$, $Z^{\mathrm{naive}}_{f,x}(s)$ or $Z^{\mathrm{top}}_{f,x}(s)$ is always a root of the Bernstein polynomial of $f$,
 as predicted by the monodromy conjecture. If $f$ has an isolated singularity at $x$, then it is even a root of order $n$, by the proof of Theorem 1 in \cite{melle-torelli-veys}.
 Beware that if $m<n$, we do {\em not} claim that the value $-\lct_x(X,\Delta)$ is a pole of order $m$ of the topological zeta function
$Z^{\mathrm{top}}_{f,x}(s)$. The Euler characteristic might be too crude as an invariant to guarantee that the residue at the expected pole is non-zero (although we do not know an explicit counterexample).
 The proof of Theorem \ref{thm-maxorder}\eqref{it:maxord} is also valid for the real parts of the poles of Igusa's local $p$-adic local zeta function at $0$ of a polynomial $f$ over $\Q$, for sufficiently large primes $p$. This can be seen from Denef's
  computation of the zeta function on a log resolution of $(X,\Delta)$ with good reduction modulo $p$ \cite{denef-euler}.

\begin{remark}\label{rem:pole} Since the Grothendieck ring $\mathcal{M}_x^{\hat{\mu}}$ is not a domain, one should specify what is meant by a pole of
 a rational function over $\mathcal{M}_x^{\hat{\mu}}$.
  The definition we use in Theorem \ref{thm-maxorder} is the following: if $Z(\LL^{-s})$ is an element of
  $$\mathcal{M}^{\hat{\mu}}_x\left[\LL^{-s},\frac{1}{1-\LL^{a-bs}}\right]_{(a,b)\in \Z\times \Z_{>0}}\subset \mathcal{M}^{\hat{\mu}}_x[[\LL^{-s}]],$$
  $s_0$ is a rational number and $m$ is a non-negative integer, then we say that
   $Z(\LL^{-s})$ has a pole at $s_0$ of order at most $m$ if we find a set $\mathscr{S}$ consisting of multisets in $\Z\times \Z_{>0}$
  such that each element of $\mathscr{S}$ contains at most $m$ elements $(a,b)$ such that $a/b=s_0$ and
  $Z(\LL^{-s})$ belongs to the sub-$\mathcal{M}^{\hat{\mu}}_x[\LL^{-s}]$-module of $\mathcal{M}^{\hat{\mu}}_x[[\LL^{-s}]]$ generated by
  $$\left\{ \frac{1}{\prod_{(a,b)\in S}(1-\LL^{a-bs})} \,|\,S\in \mathscr{S}\right\}.$$ We say that
   $Z(\LL^{-s})$ has a pole at $s_0$ of order $m$ if it has a pole at $s_0$ of order at most $m$ but not of order at most $m-1$.
 The same remark applies to $\mathcal{M}_x$.
\end{remark}

\section{The weight function and the essential skeleton}\label{sec-weight}
\sss Theorem \ref{thm-main} has an interesting reformulation
 in terms of skeleta in Berkovich spaces.
 Let $X$ be a connected smooth $k$-variety of dimension $n$, let $f:X\to \Spec k[t]$ be a regular function on $X$ and let $x$ be a closed
 point in the divisor $\Delta=(f)$. We set $R=k\llbr t\rrbr$ and $K=k\llpar t \rrpar$ and we endow $R$ with its $t$-adic topology and $K$ with its $t$-adic absolute value
 $|x|=\exp(-\mathrm{ord}_tx)$.
   We set $\cX=X\times_{k[t]}R$ and we denote by $\mX$ the formal $t$-adic completion of $\cX$. We write $\mX_K$ for the generic fiber of $\mX$ and $\mX_k=\mX\times_R k$ for its
   special fiber. Then $\mX$ is a
 separated formal scheme of finite type over $R$, and $\mX_K$ is a compact analytic domain in the $K$-analytic space $(\cX\times_R K)^{\mathrm{an}}$ associated with the $K$-variety $\cX\times_R K$. We denote by $\spe_{\cX}:\mX_K\to \mX_k$ the specialization map. A description of all these objects  in the language of birational geometry can be found in \cite{MuNi} or \cite{ni-simons}.

\sss By forgetting the $R$-structure, we can also view $\mX$ as a formal scheme over $k$, and
  consider its generic fiber $\mX_\eta$ in the sense of
  \cite[1.7]{thuillier}.
 This is an analytic
space over the field $k$ endowed with its trivial absolute value.
 It is obtained by removing from the usual generic fiber of the formal $k$-scheme
$\mX$ all the points that lie on the analytification of
the closed subscheme $\mX_k$ of $\mX$. Then $f$ defines an analytic function on $\mX_\eta$, and $\mX_K$ can be canonically identified with the subspace of $\mX_\eta$ defined by the equation $|f|=\exp(-1)$; see \cite[6.3.4]{MuNi}, where $\mX_\eta$ was denoted by $\widehat{X}_\eta$.

\sss We define the weight function
 $$\wt_{\Delta}:\mX_K\to \R\cup \{+\infty\}$$ as the restriction to  $\mX_K$ of the weight function
 $$\wt_{\mathcal{I}}:\mX_\eta\to \R\cup \{+\infty\}$$ from
 \cite[\S6.1]{MuNi}, with $\mathcal{I}=(f)$.
  This weight function is closely related to the thinness function of \cite{BoFaJo}  and the log discrepancy function of \cite{JoMu}.
 Let us briefly recall the properties of $\wt_\Delta$ that are relevant for the present paper.
   Let $h:Y\to X$ be a log resolution of the pair $(X,\Delta)$ that is an isomorphism over $X\setminus \Delta$. The dual complex
   of the strict normal crossings divisor $h^*\Delta$ can be embedded in a natural way in the $K$-analytic space
   $\mX_K$. The image of this embedding is the so-called Berkovich skeleton $\Sk(\cY)$ of $\cY=Y\times_{k[t]}R$; see for instance \cite[\S3.1]{MuNi}.
 Each prime component $E$ of $h^*\Delta$ corresponds to a vertex of $\Sk(\cY)$, and the value of the weight function $\wt_{\Delta}$ at this vertex is precisely
  the weight $\wt_{\Delta}(E)$ defined in \eqref{sss-wt}. Moreover, the weight function $\wt_{\Delta}$ is affine on every face of $\Sk(\cY)$. These properties completely determine the restriction of $\wt_{\Delta}$ to $\Sk(\cY)$.

  \sss \label{sss-reformul} We define a {\em stratum} of $h^*\Delta$ as a connected component of a non-empty intersection of a set of prime components of $h^*\Delta$. The strata of $h^*\Delta$ correspond precisely to the faces of the skeleton $\Sk(\cY)$.
    If $y$ lies in the interior of a face $\tau$ of $\Sk(\cY)$, and $\xi$ is the generic point of the stratum of $h^*\Delta$ corresponding to $\tau$, then $\spe_{\cX}(y)=h(\xi)$. We denote by $\Sk(\cY,x)$ the subspace of $\Sk(\cY)$ consisting of the points $y$ such that $x$ lies in the closure of $\{\spe_{\cX}(y)\}$. In other words, $\Sk(\cY,x)$ is the union of the faces of $\Sk(\cY)$ that correspond to strata of $h^*\Delta$ that intersect $h^{-1}(x)$. With this terminology, we can restate Theorem
  \ref{thm-main} as follows.

   \begin{theoremtag}{\ref{thm-main}, equivalent formulation}{If $\tau$ is a maximal face of $\Sk(\cY,x)$ and $\wt_{\Delta}$ is constant
  on $\tau$ with value $w$, then $w=\lct_x(X,\Delta)$.}
 \end{theoremtag}

\sss \label{sss-decrease} The embedding of the skeleton $\Sk(\cY)$ into $\widehat{\cX}_K$ has a canonical retraction
$$\rho_{\cY}:\widehat{\cX}_K\to \Sk(\cY),$$ by \cite[3.1.5]{MuNi}, with the property that $\spe_{\cX}(y)$ lies in the closure of
$\{ \spe_{\cX}(\rho_{\cY}(y))\}$ for every point $y$ in $\widehat{\cX}_K$.
 One of the most important features of the weight function
$\wt_{\Delta}$ is that it is strictly decreasing under the retraction $\rho_{\cY}$: for every point $y$ in $\widehat{\cX}_K$ we have that
$$\wt_{\Delta}(y)\geq \wt_{\Delta}(\rho_{\cY}(y))$$ and equality holds if and only if $y$ lies in $\Sk(\cY)$ (in which case $\rho_{\cY}(y)=y$).
 It is explained in \cite[6.2.2]{MuNi} how this property can be deduced from \cite[5.3]{JoMu}. Alternatively, one can use \cite[6.3.4]{MuNi} to view it as a special case of \cite[4.4.5(3)]{MuNi}.

\sss \label{sss-assumptionh} Now assume that $x$ is contained in the image of every stratum of $h^*\Delta$. Once the log resolution $h:Y\to X$ is fixed, this can always be arranged by shrinking $X$ around $x$ and shrinking $Y$ accordingly. Then $\Sk(\cY,x)=\Sk(\cY)$. The same arguments as in the proof of \cite[3.1.3]{NiXu} can be used to deduce from \cite[3.26]{thuillier} that $\rho_{\cY}$ can be extended to a strong deformation retraction of $\mX_K$ onto $\Sk(\cY)$. In particular, the embedding $\Sk(\cY)\to \mX_K$ is a homotopy equivalence. We will now construct a canonical subcomplex of $\Sk(\cY)$ that does not depend on the choice of any log resolution. The construction is motivated by the definition of the essential skeleton of a smooth and proper $K$-variety in \cite[4.6.2]{MuNi}; see Section \ref{sec-CY} below for the case of a Calabi-Yau variety. For the following definition, we do not require that $x$ is contained in the image of every stratum of $h^*\Delta$.

\begin{definition}\label{def-essentialsk}
 Denote by $S$ the subset of $\mX_K$
that consists of the points $y$ such that $x$ lies in the closure of $\{\spe_{\cX}(y)\}$.
 We define the essential skeleton of $f$ at $x$ as the set of points $y$ in $S$ such that
 the restricted weight function
 $$\wt_{\Delta}:S\to \R\cup\{+\infty\}$$ reaches its minimal value at $y$. We denote this essential skeleton by $\Sk(f,x)$.
\end{definition}

\sss From the properties of the weight function $\wt_{\Delta}$ described above, it is easy to see how the essential skeleton $\Sk(f,x)$ can be computed
 on the log resolution $h$. It follows at once from \eqref{sss-decrease} that $\Sk(f,x)$ is contained in
 $$S\cap \Sk(\cY)=\Sk(\cY,x).$$ Since $\wt_{\Delta}$ is affine on every face of $\Sk(\cY)$, the minimal value of $\wt_{\Delta}$ on $S$ is always reached at a vertex $v$ of
 $\Sk(\cY,x)$. But if $E$ is the component of $h^*\Delta$ corresponding to $v$, we have
 $$\wt_{\Delta}(v)=\wt_{\Delta}(E)$$ by definition of the weight function. Thus the minimal value of $\wt_{\Delta}$ on $S$ is precisely the log canonical threshold $\lct_x(X,\Delta)$, and $\Sk(f,x)$ is the subcomplex
 of $\Sk(\cY,x)$ spanned by the vertices that correspond to the components $E$ of $h^*\Delta$ such that $x$ lies in $h(E)$ and $\wt_{\Delta}(E)=\lct_x(X,\Delta)$. We emphasize that, by its very definition, the subspace $\Sk(f,x)$ of $\widehat{\cX}_K$ does not depend on the choice of the log resolution $h$. It is also clear from the definition that it only depends on the algebraic germ of $f$ at $x$.

\sss \label{sss-locflow} We will now describe the homotopy type of $\Sk(f,x)$ and, more generally, of the level sets of the weight function $\wt_{\Delta}$ on $\Sk(\cY,x)$. This is relevant for the study
of the motivic zeta function $Z_{f,x}(s)$, since the values of $\wt_{\Delta}$ at the vertices of $\Sk(\cY,x)$ are precisely the candidate poles that appear in the explicit formula for $Z_{f,x}(s)$ in terms of the log resolution $h$, which we recalled in \eqref{sss-formula}.
 We will need the notion of a {\em collapse} (see for instance \cite[Def.~18]{dFKX}).
 Let $D$ be a regular cell complex as in \cite[Def.~7]{dFKX}.
  For our purposes, one can think of $D$ as a finite simplicial complex where a set of vertices can span more than one face, for instance, a graph with multiple edges between pairs of vertices; in practice, $D$ will be $\Sk(\cY)$ or a subcomplex of $\Sk(\cY)$.
    Let $\tau$ be a cell in $D$ and $\sigma$ a
face of $\tau$. We say that $(\tau,\sigma)$ is a free pair if $\sigma$ is not a face of any other cell in
$D$. The elementary collapse of $(D, \tau,\sigma)$ is the regular complex obtained from $D$ by
removing the interiors of the cells $\tau$ and $\sigma$. It is clear that such an elementary collapse is a strong deformation retract of $D$.
 A sequence of elementary collapses is called a collapse.

  \begin{theorem}\label{thm-locallevel}  We assume that $\Delta$ is reduced at $x$.
 Let $h:Y\to X$ be a projective log resolution of $(X,\Delta)$ that is an isomorphism over $X\setminus \Delta$, and define
  $\cY$ and $\Sk(\cY,x)$ as above. We denote by $\Sk(\cY,x)_{\mathrm{exc}}$ the subcomplex of $\Sk(\cY,x)$ generated by the
  vertices that correspond to $h$-exceptional components of $h^*\Delta$.
 For every real number $w$ we denote by $\Sk(\cY,x)^{\leq w}$ the subcomplex of $\Sk(\cY,x)$ generated by the vertices
  where the value of $\wt_\Delta$ is at most $w$. The subcomplex $\Sk(\cY,x)^{\leq w}_{\mathrm{exc}}$ of $\Sk(\cY,x)_{\mathrm{exc}}$ is defined in the same way.

  \begin{enumerate}
  \item \label{it:exccontr} 
   If we replace $X$ by a sufficiently small \'etale neighbourhood of $x$ and restrict $h$ accordingly, then $\Sk(\cY,x)_{\mathrm{exc}}$ is contractible.
     \item \label{it:locallevel-lc}
    Assume that the pair $(X,\Delta)$ is log canonical at $x$, that is, $\lct_x(X,\Delta)=1$.
      Then there exists a collapse of $\Sk(\cY,x)$ to the essential skeleton $\Sk(f,x)$ that simultaneously collapses $\Sk(\cY,x)^{\leq w}$ to  $\Sk(f,x)$ for all $w\geq 1$.
            \item \label{it:locallevel-notlc}
  Assume that the pair $(X,\Delta)$ is not log canonical at $x$.
      Then there exists a collapse of $\Sk(\cY,x)_{\mathrm{exc}}$ to the essential skeleton $\Sk(f,x)$ that simultaneously collapses $\Sk(\cY,x)^{\leq w}_{\mathrm{exc}}$ to  $\Sk(f,x)$ for all $w\geq \lct_x(X,\Delta)$. If we replace $X$ by a sufficiently small \'etale neighbourhood of $x$, then $\Sk(f,x)$, and therefore all the spaces $\Sk(\cY,x)^{\leq w}_{\mathrm{exc}}$, are contractible.
     \end{enumerate}
\end{theorem}
\begin{proof}
\eqref{it:exccontr}   Replacing $X$ by a Zariski-open neighbourhood of $x$, we can assume that
  $\Sk(\cY,x)=\Sk(\cY)$. Denote by $\Sigma\subset X$ the discriminant locus of $h$ (with its induced reduced structure), and let 
 $h':Y'\to X$ be any log resolution of $(X,\Sigma)$ that is an isomorphism over $X\setminus \Sigma$. Since the reduced inverse image $h^{-1}(\Sigma)_{\red}$ is the union of the exceptional components of $h$, it is a strict normal crossings divisor on $Y$ and its dual complex can be identified with
 $\Sk(\cY,x)_{\mathrm{exc}}$. Thus, by Thuillier's generalization of Stepanov's theorem \cite[4.8]{thuillier}, we know that 
 $\Sk(\cY,x)_{\mathrm{exc}}$ is homotopy equivalent to the dual complex of $(h')^{-1}(\Sigma)$. Therefore, it suffices to construct (after replacing $X$ by an \'etale neighbourhood of $x$, if necessary), a log resolution $h'$ of $(X,\Sigma)$ that is an isomorphism over $X\setminus \Sigma$ and such that the dual complex of $(h')^{-1}(\Sigma)$ is contractible.

 We can always construct such a log resolution $h'$ as a composition
 $$h'=h'_0\circ \ldots \circ h'_r$$ where $h'_0$ is the blow-up of $X$ at $x$ and, for every $\ell>0$, $h'_\ell$ is a blow-up with a smooth connected center $Z_\ell$ that has transversal intersections with the exceptional divisor $F_{\ell}$ of $h'_{0}\circ \ldots \circ h'_{\ell-1}$. Replacing $X$ by a sufficiently small \'etale neighbourhood of $x$, we can also assume that the intersection of $Z_\ell$ with $F_{\ell}$ is non-empty and connected. Now it is easy to verify that the dual complex of $F_\ell$ is homotopy equivalent to that of $F_{\ell-1}$ by a slight generalization of the arguments in \cite[\S2]{stepanov} and point 9 of \cite{dFKX}. Since the dual complex of $F_1$ is a point, the dual complex of $F_{r+1}=(h')^{-1}(\Sigma)_{\red}$ is contractible.

\eqref{it:locallevel-lc} Our proof is essentially a refinement of Theorem 3 in \cite{dFKX}.
 Shrinking $X$ around $x$, we can assume that $x$ is contained in the image of every stratum of $h^*\Delta$.
 The assumption that $(X,\Delta)$ is log canonical at $x$ implies that
 $\lct_x(X,\Delta)=1$.
  We define a divisor $\Delta_0$ on $Y$ by
 $$\Delta_0=(h^*\Delta)_{\red}.$$ Then we run the relative MMP for the pair $(Y,\Delta_0)$ over $X$ with scaling of some ample divisor, as in the proof of Theorem \ref{thm-main}. The outcome is again a sequence of birational maps of $X$-schemes
\begin{equation}\label{eq:MMPseq0}
Y=Y_0\dashrightarrow Y_1 \dashrightarrow \ldots \dashrightarrow Y_m.\end{equation} For every $i$, we denote by $\Delta_i$ the pushforward of the divisor $\Delta_0$ to $Y_i$, and by $U_i$ the open subvariety of $Y_i$ consisting of the points where $Y_i$ is regular and $\Delta_i$ is a divisor with strict normal crossings. We set $\cU_{i}=U_i\times_{k[t]}R$ and we denote by $\widehat{\cU}_i$ its formal $t$-adic completion, with generic fiber $(\widehat{\cU}_i)_K$. Then the morphism $h_i:Y_i\to X$ induces an embedding $(\widehat{\cU}_i)_K\to \widehat{\cX}_K$. The $dlt$-property of the pair $(Y_i,\Delta_i)$ implies that every stratum of $\Delta_i$ intersects $U_i$, so that we can identify the dual complexes of $\Delta_i$ and the special fiber $(\cU_i)_k$, which yields a canonical homeomorphism between the skeleton $\Sk(\cU_i)$ and the dual complex of $\Delta_i$.

The divisor $K_{Y_m}+\Delta_m$ is nef over $X$, and thus the same holds for
$$D=K_{Y_m/X}+\Delta_m-h^*_m\Delta. $$
 The negativity lemma \cite[3.39]{kollar-mori} implies that $-D$ is effective, because $(h_m)_*(D)=0$ since $\Delta$ is reduced. It follows that $\wt_{\Delta}(E)\leq 1$ for every prime component $E$ of $\Delta_m$. But we assumed that $\lct_x(X,\Delta)=1$, so that $\wt_{\Delta}(E)= 1$ for every $E$. Moreover, if $E'$ is a prime component of $\Delta_0$ with $\wt_{\Delta}(E')=1$, then the definition of a $dlt$-pair implies that  the generic point of $E'$ is mapped to the locus $U_m$ in $Y_m$. This means that the vertex of $\Sk(\cY)$ corresponding to $E'$ lies in the skeleton $\Sk(\cU_m)$, because the weight function $\wt_{\Delta}$ is strictly larger than $1$ on the complement of $\Sk(\cU_m)$ in $(\widehat{\cU_m})_K$ by \eqref{sss-decrease}. Thus  $\Sk(f,x)$ is equal to $\Sk(\cU_m)$ when we view both spaces as subsets of $\widehat{\cX}_K$.

 Now choose a real number $w\geq 1$.
  We write $\Delta$ as a sum of reduced effective divisors $\Delta=A+B$ such that
   $\wt_{\omega}(E)\leq w$ for every prime component $E$ of $A$ and $\wt_{\omega}(E)> w$ for every prime component $E$ of $B$.
  We choose $\varepsilon>0$ sufficiently small and we set $\Delta'_0=A+(1-\varepsilon)B$. For every $i$ in $\{0,\ldots,m\}$, we denote by $\Delta'_i$, $A_i$ and $B_i$ the pushforwards to $Y_i$ of $\Delta'_0$, $A$ and $B$, respectively.
  The extremal ray $R_i\subset NE(Y_i/X)$ inducing $Y_i\dashrightarrow Y_{i+1}$ is also $(K_{Y_i}+\Delta'_i)$-negative, for every $i<m$.
  Moreover, $\Delta'_m=\Delta_m$ because we have seen that all the components of $B$ are contracted on $Y_m$. Thus \eqref{eq:MMPseq0} is also an MMP-sequence for $(Y,\Delta'_0)$. Observe that the components of $A$ correspond precisely to the vertices of
  $\Sk(\cY,x)^{\leq w}$. We denote by $\Sk(\cU_i)^{\leq w}$ the subcomplex of $\Sk(\cU_i)$ generated by the vertices corresponding to the components of $A_i$, i.e., the vertices
  where the value of $\wt_{\Delta}$ is at most $w$. Note that $\Sk(\cU_m)^{\leq w}=\Sk(f,x)$ since the weight function is constant with value $1$ on $\Sk(\cU_m)$.

  We claim that, for every $i$, either
   $f_i:Y_i\dashrightarrow Y_{i+1}$ does not contract any log canonical center of $A_i$, or $R_i\cdot E>0$ for some prime component $E$ of $A_i$. In the former case,
    $\Sk(\cU_i)^{\leq w}=\Sk(\cU_{i+1})^{\leq w}$. In the latter case, it follows from \cite[Thm.~19]{dFKX} that $\Sk(\cU_{i+1})^{\leq w}$ is a collapse of $\Sk(\cU_i)^{\leq w}$. Thus it suffices to prove our claim.
        The following argument is a variant of the proof of Lemma 21 in \cite{dFKX}.

 By the definition of the weight function, the divisor $K_{Y_i/X}+\Delta_i-wh_i^*(\Delta)$ can be written as
  $D_1-D_2$ such that $D_1$ and $D_2$ are effective $\Q$-divisors, $D_1$ has the same support as $B_i$, and $D_2$ is supported on $A_i$.
  Thus we can write
  \begin{equation}\label{eq-collapse0}
  0>R_i\cdot (K_{Y_i/X}+\Delta'_i-wh_i^*(\Delta))=R_i\cdot (D_1-\varepsilon B_i - D_2).
  \end{equation}
   Assume that $f_i$ contracts a log canonical center $W$ of the divisor $A_i$. By choosing $\varepsilon>0$ sufficiently small,
we can assume that the divisor $D_1-\varepsilon B_i$ is effective. It is supported on $B_i$ and therefore does not contain $W$. Thus for every curve $C$ in $\Delta_i$ through a general point of $W$, we have
    $C\cdot (D_1-\varepsilon B_i)\geq 0$. It follows that $R_i\cdot (D_1-\varepsilon B_i)\geq 0$ which implies that $R_i\cdot D_2>0$ because of \eqref{eq-collapse0}. This concludes the proof.

\eqref{it:locallevel-notlc} Our assumption that
$(X,\Delta)$ is not log canonical at $x$ implies that $\lct_x(X,\Delta)<1$, so that  $\Sk(f,x)$ is contained in $\Sk(\cY,x)_{\mathrm{exc}}$. Now the proof is similar as in case \eqref{it:locallevel-lc}, except that we define $\Delta_0$ by
$$\Delta_0=h^{-1}_*(\lct_x(X,\Delta)\cdot \Delta)+\mathrm{Ex}(h),$$ where $\mathrm{Ex}(h)$ is the reduced exceptional locus of $h$, and we set
$$D=K_{Y_m/X}+\Delta_m-h_m^*(\lct_x(X,\Delta)\cdot \Delta).$$
 Reasoning as above, we find that
$\Sk(f,x)$ is equal to $\Sk(\cU_m)_{\mathrm{exc}}$, the subcomplex of $\Sk(\cU_m)$ generated by the vertices that correspond to components of $\Delta_m$ that are contracted on $X$. Note that these are precisely the components of multiplicity one in $\Delta_m$. We denote the sum of these components by $\Delta_m^{=1}$.
 Then we can identify $\Sk(\cU_m)_{\mathrm{exc}}$ with the dual complex of $\Delta_m^{=1}$. Now another application of \cite[Thm.~19]{dFKX} shows that our MMP sequence collapses $\Sk(\cY,x)^{\leq w}_{\mathrm{exc}}$ to  $\Sk(f,x)$ for all $w\geq \lct_x(X,\Delta)$. Thus if we replace $X$ by a sufficiently small \'etale neighbourhood of $x$, then $\Sk(f,x)$ and all the spaces $\Sk(\cY,x)^{\leq w}_{\mathrm{exc}}$ are contractible, by \eqref{it:exccontr}.
\end{proof}
\begin{cor}\label{cor-defret} Suppose that $(X,\Delta)$ is log canonical at $x$. We also assume that
 there exists a projective log resolution $h:Y\to X$ of $(X,\Delta)$ that is an isomorphism over $X\setminus \Delta$ and such that $x$ is contained in the image of every stratum of $h^*\Delta$.  The latter assumption can always be guaranteed by replacing $X$ by a sufficiently small open neighbourhood of $x$. Then the essential skeleton $\Sk(f,x)$ is a strong deformation retract of $\widehat{\cX}_K$.
\end{cor}
\begin{proof}
Since $\Sk(\cY)$ is a strong deformation retract of $\widehat{\cX}_K$ by \eqref{sss-assum}, this follows at once from Theorem \ref{thm-locallevel}.
\end{proof}

\begin{remark}
Replacing $X$ by a sufficiently small \'etale neighbourhood of $x$ as in Theorem \ref{thm-locallevel}\eqref{it:locallevel-notlc} does not affect the motivic zeta function $Z_{f,x}(s)$, since, by its very definition, the zeta function only depends on the formal completion of the morphism $f$ at $x$ (in other words, on $f$ viewed as an element of the completed local ring $\widehat{\mathcal{O}}_{X,x}$).
\end{remark}

\begin{exam}
The assumption that $(X,\Delta)$ is log canonical cannot be omitted in Theorem \ref{thm-locallevel}\eqref{it:locallevel-lc} or Corollary \ref{cor-defret}, as is illustrated by the following example. Set $X=\A^2_k$ and let $x$ be the origin of $\A^2_k$. Let $C$ be an irreducible curve in $\A^2_k$ with a node at the origin, and let $L$ be a generic line in $\A^2_k$ through $x$. We set $\Delta=C+L$ and we choose a generator $f$ for the ideal sheaf $\mathcal{O}(-\Delta)$. Let $h:Y\to \A^2_k$ be the blow-up at the origin; this is a log resolution for $(\A^2_k,\Delta)$. Then $\Sk(\cY,x)=\Sk(\cY)$ has the homotopy type of a circle, while
$\Sk(f,x)$ consists only of the vertex of $\Sk(\cY)$ corresponding to the exceptional divisor of $h$ (in this vertex the weight function $\wt_\Delta$ takes the value $2/3$, whereas it is equal to $1$ at the other vertices).
\end{exam}

\begin{exam}
 The previous counterexample is somewhat artificial since we can solve the problem by passing to an \'etale neighbourhood of $x$ to break up $C$ into two irreducible components and make $\Sk(\cY,x)$ contractible. We will now give another example
 that shows that this is not always possible. This example was kindly suggested by one of the referees.

  We consider the polynomial $$f=u^{N-2}vw+v^N+w^N+u^{N+2}\in k[u,v,w]$$ where $N\geq 3$. We denote by $\Delta$ the zero locus of $f$ in $X=\mathbb{A}^3_k$ and by $x$ the origin of $\mathbb{A}^3_k$. The divisor $\Delta$ is reduced and has an isolated singularity at $x$.
   Let $Y_1\to X$ be the blow at $x$. We denote the exceptional divisor by $E_1$. 
    The strict transform $\Delta_1$ of $\Delta$ on $Y_1$ has a unique singular point $y_1$ lying over $x$, which is an $A_1$-singularity: on the blow-up chart with coordinates 
 $$(u, v'=v/u, w'=w/u),$$ the divisor $\Delta_1$ is defined by the equation 
$$v'w'+(v')^N+(w')^N+u^2=0,$$ and $y_1$ is the point $(0,0,0)$.   Let $Y_2\to Y_1$ be the blow up at $y_1$, with exceptional divisor $E_2$, and denote by $\Delta_2$ and $E_1'$ the strict  transforms of $\Delta_1$ and $E_1$ on $Y_2$, respectively.  The composed morphism $h:Y_2\to X$ is a log resolution for the pair $(X, \Delta)$.
 
 The intersection of the divisors $\Delta_1$ and $E_1$ on $Y_1$ is an irreducible curve $C$, which is defined by the equations  $$u=v'w'+(v')^N+(w')^N=0$$ on the blow-up chart with coordinates $(u,v',w')$ as above. This curve has a nodal singularity at $y_1$, so that $\Delta_2\cap E'_1\cap E_2$ consists of two points. Moreover, each pair of divisors in the set $\{\Delta_2,E_1',E_2\}$ meet along an irreducible curve. Therefore, the skeleton $\Sk(\cY_2,x)=\Sk(\cY_2)$ of $\cY_2=Y_2\times_{k[t]}R$  is homeomorphic to a 2-dimensional sphere.
  The pullback $h^*\Delta$ is given by $\Delta_2+NE'_1+(N+2)E_2$, 
 and the relative canonical divisor of $h$ is $K_{Y_2/X}=2E'_1+4E_2$.
 Thus the weights of $\Delta$ at $\Delta_2$, $E'_1$ and $E_2$ are given by $1$, $3/N$ and $5/(N+2)$, respectively.

 If $N=3$ then the pair $(X,\Delta)$ is log canonical at $x$ and $\Sk(f,x)=\Sk(\cY_2)$. The weight function is constant with value $1$ on $\Sk(\cY_2)$, and the formulas for the topological and motivic zeta functions of $f$ at $x$ in \eqref{sss-formula} show that they all have  a pole of order $3$ at $s=-1$. If, however, $N>3$, then
 the log canonical threshold of $(X,\Delta)$ at $x$ equals  $3/N$ and $\Sk(f,x)$ is the point of $\Sk(\cY_2)$ corresponding to the divisor $E'_1$.
  Thus $\Sk(f,x)$ is not homotopy equivalent to $\Sk(\cY_2)$, and replacing $X$ by an \'etale neighbourhood of $x$ will not change this situation.
\end{exam}

\section{Degenerations of Calabi-Yau varieties}\label{sec-CY}
\sss The aim of this section is to generalize Theorems \ref{thm-main} and \ref{thm-locallevel} to degenerations of Calabi-Yau varieties.
 Let
$X$ be a geometrically connected smooth projective $K$-scheme with trivial canonical sheaf, and let $\omega$ be a volume form on $X$.
  Then on the $K$-analytic space $X^{\an}$ we can again consider a weight function
 $$\wt_{\omega}:X^{\an}\to \R\cup \{+\infty\},$$ associated with the form $\omega$. This function was defined in \cite[4.4.4]{MuNi}.
  It is bounded below and the set of points in $X^{\an}$ where it reaches its minimal value is a non-empty
   compact subspace of $X^{\an}$ that we call the {\em essential skeleton} of $X$ and that we denote by $\Sk(X)$; see \cite[\S4.6]{MuNi}. This definition
   does not depend on the choice of $\omega$ because multiplying $\omega$ with an element $a\in K^{\times}$ shifts the weight function by the $t$-adic valuation of $a$.
    The essential skeleton $\Sk(X)$ was first considered by Kontsevich and Soibelman in their non-archimedean interpretation of Mirror Symmetry \cite{KS}.

  \sss We will now give a more explicit description of $\wt_{\omega}$ and $\Sk(X)$ that is sufficient to interpret the statements of our main results.
   Let $\cX$ be an $snc$-model of $X$ over $R$,
  that is,
a regular flat proper $R$-scheme endowed with an isomorphism of $K$-schemes $\cX_K\to X$ such that the special fiber $\cX_k$ is a strict normal crossings divisor.
   Then $\cX$ gives rise to
  a Berkovich skeleton $\Sk(\cX)$ in $X^{\an}$ that is canonically homeomorphic to the dual complex of $\cX_k$ (see \cite[\S3.1]{MuNi}).
   If $x$ is a vertex of $\Sk(\cX)$ corresponding to a prime component $E$ of $\cX_k$, then $$\wt_{\omega}(x)=\wt_{\omega}(E):=\frac{\nu}{N}.$$ Here $N$ is the multiplicity of $E$ in $\cX_k$ and $\nu-1$ is the multiplicity of $E$ in $\mathrm{div}_{\cX}(\omega)$, the divisor on $\cX$ associated with the rational section $\omega$ of the relative canonical line bundle $\omega_{\cX/R}$. Moreover, the weight function $\wt_{\omega}$ is affine on every face of $\Sk(\cX)$, by \cite[4.3.3]{MuNi}. These properties completely determine the
   restriction of $\wt_{\omega}$ to $\Sk(\cX)$.
   It follows from \cite[4.4.5(3)]{MuNi} that the weight function $\wt_{\omega}$ on $X^{\an}$ can only reach its minimal value at points of $\Sk(\cX)$. In other words, the essential skeleton $\Sk(X)$ is contained in $\Sk(\cX)$.
  Thus $\Sk(X)$ is the union of the faces of
  $\Sk(\cX)$ that are spanned by vertices corresponding to prime components $E$ in $\cX_k$ for which $\wt_{\omega}(E)$ is minimal (see \cite[4.5.5]{MuNi} for a generalization of this result).  The following lemma
  reduces the study of the weight function on $\Sk(\cX)$ to the case where $\cX$ is defined over an algebraic curve. We will need this reduction below to apply certain tools from the MMP.

\begin{lemma}\label{lemma-curve} Let
$X$ be a geometrically connected smooth projective $K$-scheme with trivial canonical sheaf, and let $\omega$ be a volume form on $X$.
Let $\cX$ be an $snc$-model of $X$ over $R$. Then we can always find the following objects.
 \begin{enumerate}
 \item \label{it:curve1} A smooth curve $\cC$ over $k$, a $k$-rational point $s$ on $\cC$ and a local parameter $t$ on $\cC$ at $s$, which gives rise to a $k$-morphism $\Spec R\to \cC$. We set $C=\cC\setminus \{s\}$.
 \item \label{it:curve2} A projective morphism $\cY\to \cC$ with geometrically connected fibers such that $\cY\times_{\cC}C\to C$ is smooth with trivial relative canonical sheaf, $\cY$ is regular and $\cY_s=\cY\times_{\cC}s$ is a divisor with strict normal crossings.
 \item A relative volume form $\omega'$ on $\cY\times_{\cC}C$ over $C$; with a slight abuse of notation, we will denote the base change of $\omega'$ to $\cY\times_{\cC}\Spec(K)$ with the same symbol.
 \item An isomorphism of simplicial complexes $$\Sk(\cY\times_{\cC}\Spec(R))\to \Sk(\cX)$$ that identifies the weight function $\wt_{\omega'}$ on $\Sk(\cY\times_{\cC}\Spec(R))$ with the weight function $\wt_{\omega}$ on $\Sk(\cX)$.
 \end{enumerate}
\end{lemma}
\begin{proof}
 The proof is similar to that of
\cite[4.2.4]{NiXu}. Let $N$ be a positive integer. By a standard spreading out argument combined with Greenberg Approximation, we find
 objects as in \eqref{it:curve1} and \eqref{it:curve2} together with an
 isomorphism of $R$-schemes $$\varphi:\cX\times_R R/(t^N)\to \cY\times_{\cC} \Spec(R/(t^N)).$$
 In particular, $\varphi$ induces an isomorphism of $k$-schemes $\cX_k\to \cY_s$ that we can use to identify the dual complex $\Sk(\cX)$ of $\cX_k$ with the dual complex
 $\Sk(\cY\times_{\cC}\Spec(R))$ of $\cY_s$.
   We denote by $S^+$ the spectrum of $R$ with its standard log structure and by $\cX^+$ the scheme $\cX$ endowed with the divisorial log structure associated with $\cX_k$.
  Likewise, we denote by $\cC^+$ the curve $\cC$ with the log structure induced by $s$ and by $\cY^+$ the scheme $\cY$ with the divisorial log structure associated with $\cY_s$.
 The $R$-module $$M=H^0(\cX,\omega_{\cX^+/S^+})$$ is free of rank one by \cite[7.1]{IKN}.
  Multiplying $\omega$ with $t^a$ for some integer $a$ shifts the weigh function $\wt_{\omega}$ by the constant $a$, so that we can assume that $\omega$ extends to
  a generator of $M$. But \cite[7.1]{IKN} also tells us that the $\mathcal{O}_{\cC}$-module $f_*\omega_{\cY^+/\cC^+}$ is locally free of rank one
   and that its base change to $R/(t^N)$ is canonically isomorphic to $M\otimes_{R} R/(t^N)$. Shrinking $\cC$ around $s$ if necessary, we can lift the class
   of $\omega$ in $M\otimes_R R/(t^N)$ to an element $\omega'$ of $H^0(\cY,\omega_{\cY^+/\cC^+})$ that is a relative volume form over $C$.
   If $N$ is sufficiently large, then the divisors of $\omega$ and $\omega'$, viewed as sections of the line bundles $\omega_{\cX^+/S^+}$ and $\omega_{\cY^+/\cC^+}$, respectively,
   coincide (note that both divisors are supported on $\cY_s\cong \cX_k$). Then it follows from the logarithmic interpretation of the weight function in \cite[3.2.2]{NiXu} that
   the restriction of $\wt_{\omega}$ to $\Sk(\cX)$ coincides with the restriction of $\wt_{\omega'}$ to $\Sk(\cY\times_{\cC}\Spec(R))$.
\end{proof}

  \begin{theorem}\label{thm-KS}
Let
$X$ be a geometrically connected smooth projective $K$-scheme with trivial canonical sheaf, and let $\omega$ be a volume form on $X$.
 Let $\cX$ be an $snc$-model of $X$ over $R$ and let $\tau$ be a maximal face of $\Sk(\cX)$ such that
   the weight function $\wt_{\omega}$ is constant on $\tau$ with value $w$. Then $w$ is the minimal value of $\wt_\omega$ on
   $X^{\an}$ and $\tau$ is contained in the essential skeleton $\Sk(X)$.
  \end{theorem}
\begin{proof}
By Lemma \ref{lemma-curve} we can assume that $\cX$ and $\omega$ are defined over an algebraic curve. More precisely, we may assume that $X=\cY\times_{\cC}\Spec(K)$, $\cX=\cY\times_{\cC}\Spec(R)$ and $\omega=\omega'$, where
$\cC$, $\cY$ and $\omega'$ are taken as in the statement of Lemma \ref{lemma-curve}.
  We will use similar arguments as in the proof of Theorem \ref{thm-main}.  We write $\cY_s=\sum_{i\in I}N_i E_i$, where the $E_i$ are the irreducible components of $\cY_s$. The face $\tau$ corresponds to
 a connected component $U$ of $E_J=\cap_{j\in J}E_j$ for some non-empty subset $J$ of $I$.
   The volume form $\omega'$ is
  a rational section of the relative canonical sheaf $\omega_{\cY/\cC}$ and thus defines a divisor $$\mathrm{div}_{\cY}(\omega')=\sum_{i\in I}(\nu_i-1)E_i$$ on
  $\cY$.
  Our assumption that $\wt_{\omega}$ is constant on
  $\tau$ with value $w$ is equivalent to the property that
 $$\mathrm{div}_{\cY}(\omega')+(\cY_s)_{\red}=w\cY_s$$  on some open neighbourhood of $U$ in $\cY$, since $\tau$ is a maximal face of $\Sk(\cX)$.

     We set $\Delta=(\cY_s)_{\red}$ and we run an MMP with scaling of an ample divisor for the pair $(\cY,\Delta)$ over $\cC$. This is the same as running a relative MMP for $(\cY,\Delta-\varepsilon \cY_s)$ for a sufficiently small $\varepsilon>0$ such that the latter pair is $klt$.
     By \cite[\S2]{HX}, the outcome is a series of birational maps
 $$\cY=\cY_0\dashrightarrow \cY_1 \dashrightarrow \ldots \dashrightarrow \cY_m$$ where each of the $\cY_i$ is a $\Q$-factorial normal projective $\cC$-scheme
 and each of the birational maps
 is a map of $\cC$-schemes whose restriction over $C$ is an isomorphism.
  If we set $\Delta_i=(\cY_i)_{s,\red}$  then the pair $(\cY_i,\Delta_i)$ is $dlt$, for every $i$.
   The same arguments as in the proof of Theorem \ref{thm-main} show that $\cY\dashrightarrow \cY_m$ is an open immersion on some open neighbourhood of $U$ in $\cY$.
   In our set-up, the result in \cite[3.3.4]{NiXu}
   states in particular that for every component $E$ of $\cY_s$ that is not contracted in $(\cY_m)_s$, the value $\wt_{\omega}(E)$ is the minimal weight of
  $\omega$ on $X^{\an}$. Since the components corresponding to the vertices of $\tau$ satisfy this condition, we find that $w$ is the minimal weight of
  $\omega$ on $X^{\an}$.
\end{proof}

\sss We still denote by
$X$ a geometrically connected smooth projective $K$-scheme with trivial canonical sheaf, and by $\omega$ a volume form on $X$.
 Let $\cX$ be an $snc$-model of $X$ over $R$. Then $\Sk(\cX)$ is a strong deformation retract of $X^{\an}$, by \cite[3.1.4]{NiXu}.
  In \cite[4.2.4]{NiXu} we deduced from the results in \cite{dFKX} that the essential skeleton $\Sk(X)$ is a strong deformation retract of $\Sk(\cX)$, and thus of the
  $K$-analytic space $X^{\an}$.
 It seems natural to expect that the weight function induces a flow on $X^{\an}$
   in the direction of decreasing values of $\wt_{\omega}$ that contracts $X^{\an}$ onto the subspace $\Sk(X)$ where $\wt_{\omega}$ takes its minimal value.
  Theorem \ref{thm-KS} supports this expectation. Further evidence is provided by the following theorem, which is the analog of Theorem \ref{thm-locallevel}.

  \begin{theorem}\label{thm-CY} Let
$X$ be a geometrically connected smooth projective $K$-scheme with trivial canonical sheaf, and let $\omega$ be a volume form on $X$.
 Let $\cX$ be a projective $snc$-model of $X$ over $R$.
  For every real number $w$ we denote by $\Sk(\cX)^{\leq w}$ the subcomplex of $\Sk(\cX)$ generated by the vertices
  where the value of $\wt_\omega$ is at most $w$. Then there exists a collapse of $\Sk(\cX)$ to the essential skeleton $\Sk(X)$ that simultaneously collapses $\Sk(\cX)^{\leq w}$ to  $\Sk(X)$ for all $w$ greater than the minimal value of $\wt_{\omega}$ on $X^{\an}$.
    \end{theorem}
\begin{proof}
 We can again assume that $X=\cY\times_{\cC}\Spec(K)$, $\cX=\cY\times_{\cC}\Spec(R)$ and $\omega=\omega'$, where
$\cC$, $\cY$ and $\omega'$ are taken as in the statement of Lemma \ref{lemma-curve}. Denote by $w_0$ the minimal value of $\wt_{\omega}$ on $X^{\an}$.  If we run a relative MMP of $(\cY,(\cY_{s})_{\red})$ over $\cC$ with scaling of an ample divisor, then we obtain a sequence of birational maps of $\cC$-schemes
\begin{equation}\label{eq:MMPseq}
\cY=\cY_0\dashrightarrow \cY_1 \dashrightarrow \ldots \dashrightarrow \cY_m\end{equation}
 such that $\cY_m$ is a minimal $dlt$-model.
 For ease of notation, we set $\Sk(\cY)=\Sk(\cY\times_{\cC}\Spec(R))$.

 Let $i$ be an element of $\{1,\ldots,m\}$. Even though $\cY_i\times_{\cC}\Spec(R)$ is usually no longer an $snc$-model of $X$, one can still define its skeleton by deleting the points where the special fiber
  of $\cY_i\times_{\cC}\Spec(R)$ is not a strict normal crossings divisor and taking the skeleton of the resulting open subscheme of
  $\cY_i\times_{\cC}\Spec(R)$. This is the definition
  that was given in \cite[3.1.2]{NiXu}. We will denote this skeleton by $\Sk(\cY_i)$. The $dlt$ property guarantees that $\Sk(\cY_i)$ is still canonically homeomorphic to the dual complex of the divisor $(\cY_i)_s$, since every stratum of $(\cY_i)_s$ contains a non-empty open subset of points where $(\cY_i)_s$ has strict normal crossings.
 We again denote by $\Sk(\cY_i)^{\leq w}$ the subcomplex of $\Sk(\cY_i)$ generated by the vertices where the value of $\wt_\omega$ is at most $w$. The skeleton
  $\Sk(\cY_m)$ is equal to the essential skeleton $\Sk(X)$ by \cite[3.3.4]{NiXu}, and the MMP process induces a collapse of $\Sk(\cY)$ to  $\Sk(X)$ by \cite[Cor.~22]{dFKX} (see also \cite[3.2.8]{NiXu}).   We will now show that it simultaneously collapses $\Sk(\cY)^{\leq w}$ to $\Sk(X)$ for all $w\geq w_0$; the proof is completely analogous to that of Theorem
\ref{thm-locallevel}.

  We fix $w\geq w_0$ and we write $(\cY_s)_{\red}$ as a sum of reduced effective divisors $(\cY_s)_{\red}=A+B$ such that
   $\wt_{\omega}(E)\leq w$ for every prime component $E$ of $A$ and $\wt_{\omega}(E)> w$ for every prime component $E$ of $B$.
  We choose $\varepsilon>0$ sufficiently small and we set $\Delta=A+(1-\varepsilon)B$. We denote by $\Delta_i$, $A_i$ and $B_i$ the pushforwards to $\cY_i$ of $\Delta$, $A$ and $B$, respectively, for every $i$ in $\{0,\ldots,m\}$.  Note that the vertices of $\Sk(\cY_i)^{\leq w}$ correspond precisely to the components of $A_i$.
  The extremal ray $R_i\subset NE(\cY_i/\cC)$ inducing $\cY_i\dashrightarrow \cY_{i+1}$ is  $(K_{\cY_i}+\Delta_i)$-negative, for every $i<m$.
  Moreover, $\Delta_m=(\cY_m)_{s,\red}$ because all the components of $B$ are contracted on $\cY_m$. Thus \eqref{eq:MMPseq} is also an MMP-sequence for $(\cY,\Delta)$.

 One shows in the same way as in the proof of Theorem \ref{thm-locallevel} that
 either  $\cY_i\dashrightarrow  \cY_{i+1}$ does not contract any log canonical center of $A_i$, or $R_i\cdot E>0$ for some prime component $E$ of $A_i$. In the former case,
    $\Sk(\cY_i)^{\leq w}=\Sk(\cY_{i+1})^{\leq w}$. In the latter case, it follows from \cite[Thm.~19]{dFKX} that $\Sk(\cY_{i+1})^{\leq w}$ is a collapse of $\Sk(\cY_i)^{\leq w}$. Composing all these collapses, we obtain a collapse of $\Sk(\cY)^{\leq w}$ onto $\Sk(X)$.
\end{proof}

\end{document}